\documentclass{amsart}

\usepackage{amsmath}
\usepackage{amsthm}
\theoremstyle{plain}
\newtheorem{theorem}{Theorem}[section]
\newtheorem{lemma}[theorem]{Lemma}
\newtheorem{prop}[theorem]{Proposition}

\newtheorem{cor}[theorem]{Corollary}
\newtheorem{conj}[theorem]{Conjecture}

\theoremstyle{remark}
\newtheorem{rem}[theorem]{Remark}
\newtheorem{remark}[theorem]{Remark}
\newtheorem{example}[theorem]{Example}

\usepackage{mathtools}
\usepackage{tikz-cd}
\usepackage{xurl}
\newcommand{\twomat}[4]{\begin{pmatrix} #1 & #2 \\ #3 & #4 \end{pmatrix}}

\usepackage[T1]{fontenc}
\usepackage{tabularx, colortbl}

\usepackage{
	enumitem,
	amssymb,
	stmaryrd,
	amsmath,
	xcolor,
	verbatim,
	float
}
\usepackage[expansion=false]{microtype}
\usepackage[backend=bibtex,doi=false,isbn=false,url=false,style=alphabetic]{biblatex}
\AtEveryBibitem{%
  \clearfield{note}%
}

\usepackage{graphicx}
\definecolor{ochre}{RGB}{204, 119, 34}
\usepackage[colorlinks=true,linkcolor=blue,citecolor=blue,filecolor=magenta,urlcolor=ochre]{hyperref}

\newcommand{\F}{\mathbb{F}}
\newcommand{\N}{\mathbb{N}}
\newcommand{\Q}{\mathbb{Q}}
\newcommand{\Z}{\mathbb{Z}}

\newcommand{\Zhat}{\hat{\mathbb{Z}}}
\newcommand{\E}{\mathcal{E}}
\newcommand{\Cc}{\mathcal{C}}
\newcommand{\Fc}{\mathcal{F}}
\newcommand{\Pc}{\mathcal{P}}

\newcommand{\GalQ}{\mathrm{Gal}(\bar{\Q}/\Q)}
\newcommand{\Gal}{\mathrm{Gal}}
\newcommand{\Frob}{\mathrm{Frob}}
\newcommand{\im}{\mathrm{im}}
\newcommand{\lcm}{\mathrm{lcm}}
\newcommand{\pres}[1]{\left\langle#1\right\rangle}
\newcommand{\Dc}{\mathcal{D}}
\newcommand{\onto}{\twoheadrightarrow}
\DeclareMathOperator{\rad}{\mathrm{rad}}
\newcommand{\isomto}{\overset{\sim}{\rightarrow}}

\title{The smallest invariant factor of elliptic curves, and coincidences}
\author{Alexander Milner}
\address{School of Mathematics, University of Edinburgh}
\email{a.j.c.milner@sms.ed.ac.uk}
\author{Jack Shotton}
\address{Department of Mathematical Sciences, Durham University}
\email{jack.g.shotton@durham.ac.uk}

\begin{document}
\begin{abstract}
  For an elliptic curve $\E$ over $\Q$ and a natural number $j$, Cojocaru has shown that there is an explicit constant
  $\Cc_{\E,j}$ giving (under GRH) the density of primes $p$ of good reduction such that the smallest invariant factor of
  $\E(\F_p)$ is $j$. For $\E$ without complex multiplication, we study the question of when $\Cc_{\E,j}$ is positive (a
  necessary and, on GRH, sufficient condition for there to be infinitely many such $p$), strengthening the main result of
  ~\cite{Kim}. Our arguments are group-theoretic using the image of the adelic Galois representation of
  $\E$. Experimentally, $\Cc_{\E,j}$ appears to vanish only when there is a coincidence of division fields; we document
  a number of families of such coincidences arising from abelian division fields.
\end{abstract}
\maketitle
\section{Introduction}
\label{sec:introim}
Let $\E$ be an elliptic curve over $\Q$ without complex multiplication, and let $p$ be a prime of good reduction for
$\E$. Then $\E(\F_p)$ is an abelian group that is a product of at most two cyclic groups, and so we may write
\[\E(\F_p) \cong \Z/d_{\E,p}\Z \times \Z/e_{\E,p}\Z\]
where $d_{\E,p}, e_{\E,p} \in \N$ with $d_{\E,p} \mid e_{\E,p}$. Thus $d_{\E,p} = 1$ if and only if $\E(\F_p)$ is
cyclic. We can ask how often, as $p$ varies, $d_{\E,p}$ takes a given value, and a conditional answer is given by
Cojocaru.

\begin{theorem}[\cite{Co2}, Theorem 2] \label{C2}
    Let $j$ be a positive integer. Under the generalised Riemann hypothesis (GRH) for the Dedekind zeta functions of the division fields $\mathbb{Q}(\E[n])$ of $\E$, 
    \[\#\{p \le x : d_{\E,p}=j\} \sim \mathcal{C}_{\E,j} \, Li(x)\]
    where
    \[\Cc_{\E, j} = \sum_{k=1}^\infty \frac{\mu(k)}{[\Q(\E[jk]):\Q]}.\]
    (Here $\mu$ is the M\"{o}bius function).
\end{theorem}

In fact, the same statement also holds if $\E$ has complex multiplication, and the GRH is then not required (see~\cite{murty_artins_1983}). We will only consider the case of elliptic curves without complex multiplication.

As stated, this theorem does not address the issue of whether $\mathcal{C}_{\E,j} > 0$ --- that is, whether
(conditionally) there are infinitely many $p$ with the smallest invariant factor of $\E(\F_p)$ equal to $j$.\footnote{If
  $\Cc_{\E,j} = 0$ then it is easy to show $d_{\E, p} \ne j$ for all but finitely many $p$, see Proposition~\ref{prop:Cejzeroimplication}
  below.} Indeed, it is in fact possible that $\mathcal{C}_{\E,j} = 0$, even if $j = 1$: if $\E$ has full rational
$2$-torsion, then $2 \mid d_{\E,p}$ for all $p$. For $j = 1$ this is the only obstruction:
\begin{theorem}[{\cite[Theorem 1.1]{serreResumeDesCours}}] \label{thm:CE1}Suppose that $\Q(\E[2]) \ne \Q$. Then \[\Cc_{\E, 1} > 0.\]
\end{theorem}
\begin{rem}
    The statement also holds in the case that $\E$ has complex multiplication.
\end{rem}

\begin{remark}\label{rem:CE1-proofs} The theorem is due to Serre, but the proof is unpublished.  In~\cite[Theorem~1]{guptaCyclicityGenerationPoints1990} it is shown unconditionally that, if $\Q(\E[2])\ne \Q$, then $\E(\F_p)$ is cyclic for infinitely many $p$. This implies Theorem~\ref{thm:CE1}, by
  Proposition~\ref{prop:Cejzeroimplication} below.

  Another proof of Theorem~\ref{thm:CE1} is given in \cite[Section 6]{CoMu}. This proof contains a small error which we explain and correct in Section~\ref{sec:CoMu}. Our main result also gives an alternative proof of Theorem~\ref{thm:CE1}.
\end{remark}

In~\cite{Kim}, Kim proves\footnote{Kim states that this result is conditional on the GRH, but in fact his result only
  requires Theorem~\ref{thm:CE1} which holds unconditionally if $\Cc_{\E,j}$ is defined as an infinite sum.} that the
conclusion of Theorem~\ref{thm:CE1} holds as long as $j$ is coprime to the \emph{adelic level} $m_\E$ of $\E$, the
minimal integer such that the image of the adelic Galois representation of $\E$ is the full pre-image of a subgroup of
$GL_2(\Z/m_\E\Z)$. If $A(\E)$ is the product of the primes such that the $p$-adic Galois representation associated to
$\E$ is not surjective and $N_\E$ is the conductor of $\E$, then the prime factors of $m_\E$ divide $2N_\E A(\E)$; see
Proposition~\ref{prop:Avsm}.\footnote{This definition of $A(\E)$ has the same prime factors as that in
  ~\cite{danielsSerresConstantElliptic2022} and differs from that of~\cite{cojocaruSurjectivityGaloisRepresentations2005}
  by (perhaps) factors of 30. It is unclear what definition of $A(\E)$ is used in~\cite{Kim}, but what is actually
  proved is the result as stated here in terms of the adelic level.} In Corollary~\ref{cor:j-coprime-6AE}, we prove the
following generalisation of Kim's result.
\begin{theorem} Suppose that $j$ is coprime to $2A(\E)$ and that $\Q(\E[2j]) \ne \Q(\E[j])$. Then \[\Cc_{\E, j} > 0.\]
\end{theorem}

In fact, we prove something stronger which allows us to handle some cases where $j$ and $2A(\E)$ have common
factors. See Theorem~\ref{T4b} and the examples that follow. It is an interesting question, related to `entanglements'
and `coincidences' of division fields, to classify all pairs $\E, j$ with $\Cc_{\E,j} = 0$. We give some examples in
Section~\ref{sec:coincidences}. These examples all arise from coincidences $\Q(\E[j]) = \Q(\E[jp])$ with $p = 2$ or
$p = 3$ and we conjecture that this is the only way that $\Cc_{\E,j}$ can vanish.

Coincidences of division fields have recently been studied in (for example)~\cite{DLZ}, \cite{yvonCoincidencesDivisionFields2025}, and~\cite{danielsCoincidencesNilpotentDivision2026}. In
Section~\ref{sec:coincidences} we define a $p$-coincidence to be one of the form $\Q(\E[j]) = \Q(\E[jp])$ and conjecture
that, for non-CM elliptic curves over $\Q$, $p$-coincidences occur only for $p = 2$ and $p = 3$. In
Section~\ref{sec:abelian} we document some infinite families of examples of $p$-coincidences stemming from elliptic
curves with abelian mod $n$ Galois representations for $n = 2, 3, 4$, and $8$; in Section~\ref{sec:smorgasbord} we give
a few sporadic examples that we found interesting.

We restrict attention to elliptic curves over $\Q$, although many of our results would continue to hold over number
fields containing no nontrivial abelian extension of $\Q$. The vanishing of $\Cc_{\E,1}$ for elliptic curves over number fields has
been studied in~\cite{campagna_cyclic_2023}, while coincidences over number fields are the subject of
~\cite{yvonCoincidencesDivisionFields2025}.

\subsection{Notation}
We summarise some of the key notation used throughout this paper:
\begin{itemize}
    \item We let $\E$ be an elliptic curve over $\Q$ without complex multiplication.
    \item We write $N_\E$ for the conductor of $\E$, and $m_\E$, $A(\E)$ for its \emph{adelic level} and \emph{Serre constant} (defined in Section~\ref{sec:galois} below).
    \item We let $\rho_n$ (for $n \in \N$) and $\rho$ be the mod $n$ and adelic Galois representations associated to $\E$ (defined in Section~\ref{sec:galois} below). 
    \item We write $\mu$ for the M\"{o}bius function.
    \item For a positive integer $b$, let $\rad(b)$ be its radical and let $\langle b \rangle$ be the set of integers all of whose prime factors divide $b$.
    \item For a group $G$ let $[G,G]$ be its commutator subgroup.
    \item The cardinality of the group $GL_{2}(\mathbb{Z}/n\mathbb{Z})$ is given by \[\psi(n) = |GL_2(\Z/n\Z)| = n^4 \prod_{p|n} \left(1-\frac{1}{p}\right) \left(1-\frac{1}{p^2}\right).\]
\end{itemize}

\subsection{Acknowledgments and data availability}

AM was funded by a London Mathematical Society Undergraduate Research Bursary [grant reference URB-2023-11], by Durham University and by a studentship from the University of Edinburgh. In the final stages of preparing this article, JS was supported by the Engineering and Physical Sciences Research Council [grant number UKRI1020].

We thank Harris Daniels for useful conversations on the topic of this paper.

Supporting code is available at
\url{https://github.com/jackshotton/invariant-factors-public}. Much of this code was produced with the assistance of GPT-5.4, including that used for the analysis of the explicit example in Section~\ref{sec:abelian-2}. 

All data used in this work is freely available from the L-functions and Modular Forms Database~\cite{lmfdb}.

\subsection{Compliance with ethical standards}

We declare no conflicts of interest.

\section{Galois images}
\label{sec:galois}
Associated to $\E$, we have the \emph{adelic Tate module} 
\[T(\E) = \varprojlim_{n} \E[n].\]
It is a free $\Zhat$-module of rank 2, and we will fix a choice of basis. We then obtain a continuous representation
\[\rho : \GalQ \to GL_2(\Zhat)\]
where $\GalQ$ is the absolute Galois group of $\Q$. Let \[H = \im(\rho) \subset GL_2(\Zhat).\]

For each $n \in \N$, let 
\[\rho_n : \GalQ \to GL_2(\Z/n\Z)\]
be the mod-$n$ reduction of $\rho$. We introduce the following notation: 
\begin{itemize}
    \item $\bmod_n : GL_2(\Zhat) \to GL_2(\Z/n\Z)$ is the reduction mod $n$ map;
    \item $\Gamma_n = \ker(\bmod_n)$;
    \item for $G \subset GL_2(\Zhat)$, $G \bmod n = \bmod_n(G)$; 
\end{itemize}
If $G \subset GL_2(\Zhat)$ is an open subset, the \emph{level} of $G$ is the least integer $m$ such that $\Gamma_m \subset G$. 
We note now, and use without comment, that if $A, B \subset GL_2(\Zhat)$ are subsets that both contain the same $\Gamma_m$, then $A\cup B \bmod m = A\bmod m \cup B \bmod m$, and similarly for intersections.

As $\E$ does not have complex multiplication, $H$ is open by~\cite[Théorème 3]{Ser}, and we define the \emph{adelic level} $m_{\E}$ of $\E$ to be the level of $H$. 
For $k \in \N$ we have
\[[\Q(\E[k]):\Q] = |\rho_k(\GalQ)| = |H \bmod k| = |H/H\cap \Gamma_k|.\]
As in Theorem~\ref{C2}, for $j \in \N$ we define
\[\Cc_{\E, j} = \sum_{k=1}^\infty \frac{\mu(k)}{[\Q(\E[jk]):\Q]}.\]
The sum is absolutely convergent, as Serre's open image theorem implies that $[\Q(\E[k]):\Q]  \asymp \psi(k) \asymp k^4$ for $k$ squarefree.

Let $m\in \N$ be any integer divisible by the adelic level $m_\E$ (usually we will have $m = m_\E$).

\begin{lemma}\label{lem:FK} Suppose that $k \in \N$ and that $p$ is a prime such that $v_p(k) \ge v_p(m)$. Then 
\[|H\bmod pk| = |H\bmod k| \cdot \begin{cases} p^4 &\text{ if $p \mid k$} \\ \psi(p) & \text{if $p\nmid k$}.\end{cases}\]
\end{lemma}
\begin{proof}
    This is the \emph{Claim} in~\cite{FK14} Section~7.
\end{proof}

\begin{prop}\label{prop:cej-factor}
Let $j \in \N$. Let $S$ be the set of prime numbers such that $v_p(j) < v_p(m)$. Then:
\begin{equation}\label{eq:cej-factor}\Cc_{\E,j} = \sum_{k \in \pres{S}} \frac{\mu(k)}{|H\bmod jk|} \prod_{p\not\in S, p \mid j}\left(1 - \frac{1}{p^4}\right)\prod_{p \not\in S, p \nmid j}\left(1 - \frac{1}{\psi(p)}\right) .\end{equation}
Furthermore, $\Cc_{\E, j} > 0$ if and only if $\Cc_{\E, \gcd(j,m)} > 0$.
\end{prop}
\begin{proof}
    If $n \in \N$ is squarefree and of the form $kn'$, where all the prime factors of $k$ are in $S$ and all the prime factors of $n'$ are not in $S$, then by Lemma~\ref{lem:FK} we have
    \[\frac{\mu(n)}{|H\bmod jn|} = \frac{\mu(k)}{|H \bmod jk|}\prod_{p\mid n', p\mid j} \frac{-1}{p^4}\prod_{p \mid n', p \nmid j}\frac{-1}{\psi(p)}.\]
    The factorisation~\eqref{eq:cej-factor} follows.

    As $\prod_{p}\left(1 - \frac{1}{\psi(p)}\right)$ is easily shown to be positive, we see that $\Cc_{\E,j} > 0$ if and only if \[\sum_{k \in \pres{S}} \frac{\mu(k)}{|H\bmod jk|} > 0.\]
    If $p$ is a prime that divides $\frac{j}{\gcd(j,m)}$ then $p \not \in S$ and so, for any $k \in \pres{S}$, $v_p(jk) = v_p(j) > v_p(m)$. So, by (repeated application of) Lemma~\ref{lem:FK}, \[|H \bmod jk| = A|H \bmod \gcd(j,m)k|\] for a positive integer $A$ independent of $k$. Thus 
    \[\sum_{k \in \pres{S}} \frac{\mu(k)}{|H\bmod jk|} = \frac{1}{A}\sum_{k \in \pres{S}} \frac{\mu(k)}{|H\bmod \gcd(j,m)k|}.\]
    Since $S = \{p : v_p(j) < v_p(m)\} = \{p : v_p(\gcd(j,m)) < v_p(m)\}$, this is positive if and only if $\Cc_{\E, \gcd(j,m)} > 0$ and we obtain the last part of the proposition.    
\end{proof}

In particular, if we are interested in when $\Cc_{\E,j} > 0$, we may restrict attention to integers $j$ that divide the adelic level $m_\E$.

\begin{prop}\label{prop:Cejiff} Suppose that $j \mid m$ and let $S = \{p : p \mid m/j\}$. Then $\Cc_{\E, j} = 0$ if and only if 
\[H \cap \Gamma_j \bmod m = \bigcup_{p \in S}H \cap \Gamma_{jp} \bmod m.\]
\end{prop}
\begin{proof}
    By Proposition~\ref{prop:cej-factor}, $\Cc_{\E, j} = 0$ if and only if 
    \[\sum_{k \in \pres{S}} \frac{\mu(k)}{|H\bmod jk|} = 0.\]
    Multiplying this by $|H\bmod j|$ we see that this is equivalent to 
    \[\sum_{k \in \pres{S}} \mu(k)\left|\frac{H\cap \Gamma_{jk}}{H\cap \Gamma_j}\right| = 0\]
    and, since each $jk \mid m$, this is equivalent to 
    \[\sum_{k \in \pres{S}} \mu(k)\frac{|H\cap \Gamma_{jk} \bmod m|}{|H\cap \Gamma_j \bmod m|} = 0.\]
    For each $k \in \pres{S}$ squarefree, $H\cap \Gamma_{jk} \bmod m = \bigcap_{p \mid k} H \cap \Gamma_{jp} \bmod m$ and so, by the inclusion-exclusion principle, 
        \[\sum_{k \in \pres{S}} \mu(k)\frac{|H\cap \Gamma_{jk}\bmod m|}{|H\cap \Gamma_j \bmod m|} = 1 - \frac{1}{|H\cap \Gamma_j \bmod m|}\left| \bigcup_{p \in S} H \cap \Gamma_{jp}\bmod m\right|.\]
    The proposition follows.
\end{proof}

\begin{cor}\label{cor:meqj}
    Suppose that $m_\E \mid j$. Then $\Cc_{\E, j} > 0$.
\end{cor}
\begin{proof}
   We may assume that $j = m_\E$ and take $m = m_\E$ in Proposition~\ref{prop:Cejiff}. Then the set $S$ is empty and so
   $\Cc_{\E, j} > 0$.
\end{proof}

\begin{cor} \label{cor:revdir}
    Let $j \in \N$. If there exists a prime $p$ such that $\Q(\E[j])=\Q(\E[pj])$, then $\Cc_{\E,j}=0$.
\end{cor}
\begin{proof}
Take $m = m_\E jp$. As 
\[H \cap \Gamma_j \bmod m = \Gal(\Q(\E[m])/\Q(\E[j])),\] and similarly for $pj$, the hypothesis implies that \[H \cap \Gamma_{j} \bmod m = H \cap \Gamma_{pj} \bmod m\]
and so $\Cc_{\E,j} = 0$ by Proposition~\ref{prop:Cejiff}.
\end{proof}

Based on numerical evidence, we make the following conjecture.

\begin{conj}\label{conj:one-prime}
    Suppose that $\E$ is a non-CM elliptic curve over $\Q$ and that $\Cc_{\E,j} = 0$. Then there exists a prime $p$ such
    that $\Q(\E[j]) = \Q(\E[pj])$.
\end{conj}

\begin{remark}
  Equalities $\Q(\E[a]) = \Q(\E[b])$ are called `coincidences' in~\cite{DLZ}. For more discussion, see
  Section~\ref{sec:coincidences} below.
\end{remark}

\begin{prop}\label{prop:Cejzeroimplication}
  Suppose that $j \in \N$ and that $\Cc_{\E,j} = 0$. If $l$ is a prime of good reduction for $\E$, then $d_{\E,l} \ne j$
  unless $j = 1$ and $l = 2$.
\end{prop}
\begin{proof}
  Let $l$ be a prime of good reduction such that $d_{\E, l} = j$. In particular, $(\Z/j\Z)^2 \subset \E(\F_l)$ which
  implies that $l \nmid j$. Suppose that $\Cc_{\E,j} = 0$.
  
 Let $m = m_\E j$. Then, by Proposition~\ref{prop:Cejiff}, 
 \[H \cap \Gamma_j \bmod m = \bigcup_{p \mid m_\E} H \cap \Gamma_{jp} \bmod m.\] Let $K = \Q(\E[m])$ so that
 $\Gal(K/\Q) = H \bmod m$. Fix a prime of $K$ above $l$ and let $I_l$ and $\tilde{\Frob}_l$ be the inertia
 group and lift of a Frobenius element.  Since $l$ is a prime of good reduction and $l \nmid j$, $I_l \subset H \cap
 \Gamma_j\bmod m$
 and $(\Z/j\Z)^2\subset \E(\F_l)$ if and only if $\tilde{\Frob}_l \in H \cap \Gamma_j \bmod m$. As $\Cc_{\E,j} = 0$, we
 have $\tilde{\Frob}_l \in H \cap \Gamma_{jp} \bmod m$ for some $p$.

 If $p = l$, suppose for now that $l$ is odd. Then $\Q(\E[jl])/\Q(\E[j])$ is ramified at primes above $l$ (since
 $\Q(\E[jl])$ contains $\zeta_l$ but $\Q(\E[j])$ is unramified at $l$). We may therefore choose $\tilde{\Frob}_l$ so
 that its image in $H\cap \Gamma_{j}/H\cap \Gamma_{jl}$ is nontrivial. For this choice of $\tilde{\Frob}_l$, we must
 therefore have $\tilde{\Frob}_l \in H \cap \Gamma_{jp} \bmod m$ for some $p \ne l$.

 But now $(\Z/jp\Z)^2 \subset \E(\F_l)$ and so $jp \mid d_{\E, l}$. Therefore, if $j \mid d_{\E,l}$ then
 $jp \mid d_{\E,l}$ for some $p$ and so $d_{\E,l} \ne j$.

 It remains to deal with the case $l = 2$. Since $|\E(\F_2)| \le 5$ and $|\E(\F_2)[2]| \le 2$, we must have $j = 1$. (By
 Theorem~\ref{thm:CE1}, this case occurs exactly when $\E$ has good reduction at 2 and full rational $2$-torsion.)
\end{proof}

Associated to $\E$ we also have the integers $A(\E)$, the product of all primes $p$ such that the $p$-adic Galois representation $\varprojlim \rho_{p^n}$ is not surjective, and $N_\E$, the conductor of $\E$. In the next proposition, we relate primes dividing $A(\E)$, $m_\E$ and $N_\E$. It may be well-known to experts, but we could not find a reference in the generality stated (including primes $p = 3, 5$).

\begin{prop}\label{prop:Avsm} Let $j$ be a positive integer and $p$ be a prime.
\begin{enumerate}
    \item If $p \mid A(\E)$ then $p \mid m_\E$.
    \item If $p \mid m_\E$ then $p \mid 2 N_\E A(\E)$.
    \item If $\gcd(j, 2A(\E)) = 1$ then $\rho_j$ is surjective.
\end{enumerate}
\end{prop}
\begin{proof}
\begin{enumerate}
    \item If $p \nmid m_\E$ then $\Gamma_{m_\E} \to GL_2(\Z_p)$ is surjective and so $p \nmid A(\E)$. 
    
    \item Suppose that $p \nmid 2N_\E A(\E)$. We want to show that $p \nmid m_\E$. Let $m_\E = p^rm'$ with $p\nmid m'$. Then 
    \[H \bmod m_\E = \im(\rho_{m_\E}) \subset GL_2(\Z/p^r\Z) \times GL_2(\Z/m'\Z)\]
    where the projection to the first factor is surjective as $p \nmid A(\E)$.
   Let $N\times \{I\}$ be the kernel of $H\bmod m_\E \to H\bmod m'$. The conclusion $p \nmid m_\E$ follows if we can show that $N = GL_2(\Z/p^r\Z)$; so suppose that $N$ is a proper subgroup.
   
   We have that
   \[\Q(\E[p^r])^N = \Q(\E[p^r]) \cap \Q(\E[m'])\subset \Q(\E[m'])\]
    is unramified at $p$, as $p \nmid N_\E$ and $p\nmid m'$. As 
    \[\Q(\E[p^r])^{SL_2(\Z/p^r\Z)} = \Q(\zeta_{p^r})\]
    is totally ramified at $p$, we see that $N \cdot SL_2(\Z/p^r\Z)=GL_2(\Z/p^r\Z)$.

    Suppose first that $GL_2(\Z/p^r\Z)/N$ admits a nontrivial abelian quotient. Then there is $K$ with $N \subset K \subset GL_2(\Z/p^r\Z)$ and $GL_2(\Z/p^r\Z)/K$ abelian. By~\cite[Lemma 7.7]{zywinaExplicitOpenImages2024} and the assumption $p > 2$, $SL_2(\Z/p^r\Z)\subset K$. Then 
    \[N\cdot SL_2(\Z/p^r\Z) = GL_2(\Z/p^r\Z) \subset N\cdot K = K,\] contradicting that the quotient was nontrivial. 

    Therefore $GL_2(\Z/p^r\Z)/N$ has a simple nonabelian quotient $S$, which implies that $p \ge 5$. Let $K\supset N$ be
    the corresponding maximal normal subgroup of $GL_2(\Z/p^r\Z)$. If $K \bmod p = GL_2(\F_p)$ then, as $p \ge 5$,
    ~\cite[IV~Lemma~3]{serre_abelian_1968} shows that $SL_2(\Z/p^r\Z) \subset K$, contradicting that $S$ is
    nonabelian. So
    \[GL_2(\Z/p^r\Z)/K \cong GL_2(\F_p)/(K \bmod p).\] Since $S$ is simple, we have
    $(K\bmod p)\cdot \F_p^\times = K\bmod p$ or $GL_2(\F_p)$. In the second case, $S$ would be abelian. So we are in the
    first case, and $S$ is a quotient of $PGL_2(\F_p)$. The natural map of simple groups $PSL_2(\F_p) \to S$ must then
    be an isomorphism, whence $PGL_2(\F_p) \cong PSL_2(\F_p) \times C_2$; this is false, so we have a contradiction.
    
    \item If $\gcd(j,30A(\E)) = 1$ then this is due to Serre, see the appendix to~\cite{cojocaruSurjectivityGaloisRepresentations2005}. 

    Write $\bar{H}$ for $H \bmod j$. We argue by induction on the largest prime factor $p$ of $j$. If $j$ is a power of 3 then the result is true by definition of $A(\E)$. Suppose, then, that $p > 3$ and write $j = p^rj'$ where $r \ge 1$ and $j'$ is coprime to $p$. The inductive hypothesis is that $\rho_{j'}$ is surjective. 

    Then $\bar{H}$ is a subgroup of $GL_2(\Z/p^r\Z) \times GL_2(\Z/j'\Z)$ such that $\bar{H}$ surjects onto each factor and such that $\det|_{\bar{H}}$ is surjective.
    By Goursat's lemma, there are normal subgroups $M \subset GL_2(\Z/p^r\Z)$ and $N \subset GL_2(\Z/j'\Z)$ such that $M \times N \subset \bar{H}$ and $\bar{H}/(M\times N)$ is the graph of an isomorphism 
    \[\phi : GL_2(\Z/p^r\Z)/M \isomto GL_2(\Z/j'\Z)/N.\]
    If $SL_2(\Z/p^r\Z)\subset M$, then $GL_2(\Z/j'\Z)/N$ is abelian and so $SL_2(\Z/j'\Z) \subset N$ by (for example)~\cite[Lemma~7.7]{zywinaExplicitOpenImages2024}. The surjectivity of $\det|_{\bar{H}}$ now forces $M = GL_2(\Z/p^r\Z)$ and $N = GL_2(\Z/j'\Z)$, as required.

    So we may assume that $SL_2(\Z/p^r\Z) \not \subset M$ and $\phi$ restricts to an injective homomorphism from the nontrivial group $SL_2(\Z/p^r\Z)/(M \cap SL_2(\Z/p^r\Z))$ to a normal subgroup of  $GL_2(\Z/j'\Z)/N$. It follows from~\cite[IV Lemmas~2 and~3]{serre_abelian_1968}, using $p \ge 5$, that the only simple quotient of $SL_2(\Z/p^r\Z)$ is $PSL_2(\F_p)$. Thus $PSL_2(\F_p)$ occurs as a subquotient of $GL_2(\Z/j'\Z)$. But either $p > 5$ and $j'$ is coprime to $p$, or $p = 5$ and $j'$ is a power of three. In either case, this is impossible (see~\cite[IV p.\ 25]{serre_abelian_1968}).\qedhere
\end{enumerate}
\end{proof}
\section{A positivity criterion}
\label{sec:positivity}

We now prove our main positivity criterion for $\Cc_{\E,j}$. To ease notation we redefine $H$ to be the image of the mod $m_\E$ Galois representation:
\[\rho_{m_\E} : \GalQ \onto H \subset GL_2(\Z/m_\E\Z).\]
If $k \mid m_\E$, then define 
\[\Dc_k : H \to (\Z/k\Z)^\times\]
to be the composition of determinant map with reduction modulo $k$. Since the determinant of $\rho_{m_\E}$ is the cyclotomic character, we can also see $\Dc_k$ as the restriction map 
\[\Gal(\Q(\E[m_\E])/\Q) \to \Gal(\Q(\zeta_k)/\Q).\]
\begin{lemma}\label{lem:cyclotomic-lemma}
    Suppose that $j, k \in \N$ with $jk \mid m_\E$. Then 
    \[\Q(\E[j]) \cap \Q(\zeta_{jk}) = \Q(\zeta_j)\]
    if and only if 
    \[\Dc_{jk}(H \cap \Gamma_j) = (1 + j\Z/jk\Z) \subset (\Z/jk\Z)^\times.\]
\end{lemma}
\begin{proof}
    The image of the map 
    \[\Dc_{jk} : H  \cap \Gamma_j = \Gal(\Q(\E[m_\E])/\Q(\E[j])) \to \Gal(\Q(\zeta_{jk})/\Q) = (\Z/jk\Z)^\times\]
    is $\Gal(\Q(\zeta_{jk})/(\Q(\zeta_{jk}) \cap \Q(\E[j])))$. The lemma follows, as $\Gal(\Q(\zeta_{jk})/\Q(\zeta_j)) = 1 + j\Z/jk\Z$.
\end{proof}

\begin{theorem} \label{T4b}
    Let $j \mid m_\E$ be a natural number, let $R$ be the product of $j$ with the odd prime factors of $m_\E/j$.
    
    If $\mathbb{Q}(\E[j]) \cap \mathbb{Q}(\zeta_{R}) =\mathbb{Q}(\zeta_{j}) $ and $\mathbb{Q}(\E[2j]) \neq \mathbb{Q}(\E[j])$, then $\mathcal{C}_{\E,j}>0$.
\end{theorem}
\begin{proof}
    Write $m_\E/j = \prod_{i=1}^r p_i^{a_i}$ with $a_i \in \N$ and $p_i$ distinct primes. Let $m_i = v_{p_i}(m_\E)$ and $j_i = v_{p_i}(j)$, so that $a_i = m_i - j_i$. 

    For any integer $n\ge 0$, set 
    \[K_n(p_i) = \begin{cases}
        I + p_i^nM_2(\Z/p_i^{m_i}\Z) & \text{if $n > 0$, and}\\
        GL_2(\Z/p_i^{m_i}\Z) & \text{if $n = 0$.}
    \end{cases}\]
    Similarly, set 
     \[U_n(p_i) = \begin{cases}
        1 + p_i^n\Z/p_i^{m_i}\Z & \text{if $n > 0$, and}\\
        (\Z/p_i^{m_i}\Z)^\times & \text{if $n = 0$.}
    \end{cases}\]

    Then $H \cap \Gamma_j \subset \prod_{i=1}^r K_{j_i}(p_i)$ and to show that $\Cc_{\E,j} > 0$ it is enough, by Proposition~\ref{prop:Cejiff}, to show that 
    \[H \cap \Gamma_j \ne \bigcup_{i=1}^r H \cap \Gamma_{p_ij}.\]
    This is equivalent to showing that there is $h \in H \cap \Gamma_j$ such that $\pi_i(h) \ne e$ for all $1 \le i \le r$, where 
    \[\pi_i : H \cap \Gamma_j \to K_{j_i}(p_i)/K_{j_i + 1}(p_i)\]
    is the projection.

    Suppose first that $2 \not \in \{p_1, \ldots, p_r\}$. By our hypothesis and Lemma~\ref{lem:cyclotomic-lemma}, 
    \[\det : H \cap \Gamma_j \to \prod_{i=1}^r U_{j_i}(p_i)/U_{j_i+1}(p_i)\]
    is surjective. Since each factor in this product is nontrivial, there is $h \in H \cap \Gamma_j$ whose image in every factor is not the identity. This is the element $h$ we need.

    Otherwise, suppose without loss of generality that $p_1 = 2$. The image $G_1$ of $H\cap \Gamma_j$ in $K_{j_1}(2)/K_{j_1 + 1}(2)$ is isomorphic to $\Gal(\Q(\E[2j])/\Q(\E[j]))$, and hence nontrivial by hypothesis.

    Let $G_2 = \prod_{i=2}^rU_{j_i}(p_i)/U_{j_i + 1}(p_i)$ and consider the homomorphism 
    \[\theta : H \cap \Gamma_j \to G_1 \times \left(\prod_{i=2}^r U_{j_i}(p_i)/U_{j_i + 1}(p_i)\right) = G_1 \times G_2\]
    and let $T = \im(\theta)$. It is enough to show that there is $(t_1, \ldots, t_r) \in T$ such that no $t_i$ is the identity.
    
    The projection $p_1$ of $T$ onto $G_1$ is surjective by definition, while the projection $p_2$ of $T$ onto $G_2$ is
    surjective by Lemma~\ref{lem:cyclotomic-lemma}. Let $N_1 = \ker(p_2) \subset G_1$ and let
    $N_2 = \ker(p_1) \subset G_2$. By Goursat's Lemma, $N_1 \times N_2 \subset T$ and the image of $T$ in
    $G_1/N_1 \times G_2 / N_2$ is the graph of an isomorphism $\phi : G_2/N_2 \to G_1/N_1$.

    Suppose that $x = (x_2, \ldots, x_r) \in G_2$ with $x_i \ne e$ for every $2 \le i \le r$. Then $\phi(x) = e$
    otherwise $(\phi(x), x_2, \ldots, x_r)$ would be the required element $t$. So $N_2$ contains every element of this
    form.

    Now let $x = (x_2, \ldots, x_r) \in G_2$ be arbitrary. For each $2 \le i \le r$ we may choose $y_i$ and $z_i$ in
    $U_{j_i}(p_i)/U_{j_i+1}(p_i)$ such that $y_i z_i = x_i$ and $y_i, z_i \ne e$, \emph{unless} $p_i = 3$, $j_i = 0$ and
    $x_i = 2$. Suppose for now that this does not happen. Then taking $y = (y_2, \ldots, y_r)$ and
    $z = (z_2, \ldots, z_r)$, we have $x = yz$ and, by the previous paragraph, $y, z \in N_2$. So $x \in N_2$. But then,
    as $x$ was arbitrary, $N_2 = G_2$ and so $T = G_1 \times G_2$, in which case we can easily choose $t \in T$ that is
    not the identity in any factor.

    If there is some $i$ with $p_i = 3$, $j_i = 0$ and $x_i = 2$ then, without loss of generality, $i = 2$. We then take
    $w = (2,w_2, \ldots, w_r)$, $y = (2, y_2, \ldots, y_r)$ and $z = (2, z_2, \ldots, z_r)$ where all
    $w_i, y_i, z_i \ne e$ and $w_iy_iz_i = x_i$ (for example, take $w_i \ne e$ arbitrary and then find $y_i, z_i \ne e$
    such that $y_i z_i = x_i w_i^{-1}$). Then $w, y, z \in N_2$ and $x = wyz$; we conclude as before.
\end{proof}

\subsection{Corollaries}
\label{sec:corollaries}
As our first application of Theorem~\ref{T4b}, we give a proof of 
 Theorem~\ref{thm:CE1} of the introduction. This is due to Serre (see \cite{serreResumeDesCours}); see Remark~\ref{rem:CE1-proofs} for a discussion of the proofs in the literature.
\begin{cor}\label{cor:CE1} If $\Q(\E[2]) \ne \Q$ then $\Cc_{\E,1} > 0$.
\end{cor}
\begin{proof} Immediate from Theorem~\ref{T4b}, taking $j = 1$.
\end{proof}
\begin{cor}\label{cor:j-abelianisation} If $j$ is odd, $\Q(\E[2j]) \ne \Q(\E[j])$, and the image $\bar{H} = H \bmod j$ of $\rho_j : \GalQ \to GL_2(\Z/j\Z)$ satisfies 
\[[\bar{H}, \bar{H}] = \bar{H} \cap SL_2(\Z/j\Z),\]
then $\Cc_{\E,j} > 0$.
\end{cor}
\begin{proof}
    The maximal abelian subextension of $\Q(\E[j])$ is equal to 
    \[\Q(\E[j])^{[\bar{H},\bar{H}]} = \Q(\E[j])^{\bar{H} \cap SL_2(\Z/j\Z)} = \Q(\zeta_j).\]
    In particular, Theorem~\ref{T4b} applies.
\end{proof}
\begin{remark}
    If $j$ is odd, $4 \nmid m_\E$, and $9 \nmid m_\E/j$, then one can show that the conditions in Corollary~\ref{cor:j-abelianisation} are equivalent to those in Theorem~\ref{T4b}.
\end{remark}


\begin{example} Let $\E : y^2+y=x^3+x^2+x$ be the elliptic curve with LMFDB label \href{https://www.lmfdb.org/EllipticCurve/Q/19/a/3}{19.a3}. Then $\rho_3$ has image 
\[\left\{\begin{pmatrix}1 & \star  \\ 0  & \star \end{pmatrix}\right\}\]
and $[\Q(\E[6]) : \Q(\E[3])] = 6$. Thus the hypotheses of Corollary~\ref{cor:j-abelianisation} apply to show that $\Cc_{\E, 3} > 0$. There are infinitely many similar examples, see~\cite{zywina_possible_2015}.
\end{example}
We now obtain a simple generalisation of the main theorem of~\cite{Kim}, removing the condition that $\gcd(j, N_\E) = 1$.
\begin{cor} \label{cor:j-coprime-6AE} Suppose that $\gcd(j, 2A(\E)) = 1$ and that $\Q(\E[2j]) \ne \Q(\E[j])$. Then $\Cc_{\E,j} > 0$.
\end{cor}
\begin{proof}
    This follows from Corollary~\ref{cor:j-abelianisation} and Proposition~\ref{prop:Avsm}, as the derived subgroup of $GL_2(\Z/p^a\Z)$ is $SL_2(\Z/p^a\Z)$ for $p \ge 3$ by (for example)~\cite[Lemma~7.7]{zywinaExplicitOpenImages2024}.
\end{proof}

In Theorem~\ref{T4b}, the condition $\Q(\E[2j]) \ne \Q(\E[j])$ is necessary, by Corollary~\ref{cor:revdir}. However, the other condition is certainly not necessary, as the following example shows.
\begin{example}
    Let $\E : y^2+xy+y=x^3+x^2-3x+1$ be the elliptic curve with LMFDB label \href{https://www.lmfdb.org/EllipticCurve/Q/50/b/3}{50.b3}. Then the image of $\rho_3$ is the full group of upper triangular matrices, $R = 15$, and $\Q(\E[3]) \cap \Q(\zeta_R)$ is the quadratic extension of $\Q(\zeta_3)$ in $\Q(\zeta_{15})$. However, from the information in~\cite{lmfdb}, we have 
    \[\frac{1}{|H\bmod 3|} - \frac{1}{|H\bmod 6|} - \frac{1}{|H\bmod 15|} + \frac{1}{|H\bmod 30|} > \frac{1}{12} - \frac{1}{72} - \frac{1}{120} > 0\]
    and so $\Cc_{\E,3} > 0$ by Proposition~\ref{prop:cej-factor}.
\end{example}
On the other hand, here is an example where the even $j$ are the crucial ones to consider:
\begin{example} Let $\E: y^2=x^3-3x+4$ with LMFDB label \href{https://www.lmfdb.org/EllipticCurve/Q/5184/j/1}{5184.j1}. Then $m_\E = 4$, so the first condition in Theorem~\ref{T4b} is vacuous. As $\Q(\E[4]) \neq \Q(\E[2]) \neq \Q$, we see that $\Cc_{\E,j}>0$ for all $j \mid m_\E$ and hence for all positive integers $j$.
\end{example}

Jones~\cite{jones_almost_2010} shows that 100\% of elliptic curves (ordered by naive height) are ``Serre curves'', meaning that the index of $\im(\rho)$ in $GL_2(\hat{\Z})$ is two (as small as possible). As we completely understand the adelic image in this case, the analysis is straightforward.
\begin{prop}\label{prop:serre} Let $\E$ be a Serre curve. Then $\Cc_{\E,j} > 0$ for all $j$.
\end{prop}
\begin{proof}
    We may assume (Proposition~\ref{prop:cej-factor}) that $j \mid m_\E$ and write $H = \im(\rho_{m_\E})$.
    Let $\Delta_\E$ be the discriminant of $\E$ and let $\chi$ be the quadratic (or trivial, if $\Delta_\E$ is square) Dirichlet character of conductor $m$ associated to the extension $\Q(\sqrt{\Delta_\E})/\Q$, regarded as a character of $(\Z/m\Z)^\times$.  Then by~\cite[section 3]{jones_averages_2009} we have that $m_\E = \lcm(2,m)$ and
    \[H = \ker(\epsilon\cdot(\chi\circ \det)),\]
    where $\epsilon$ is the nontrivial quadratic character of $GL_2(\F_2)$, inflated to $GL_2(\Z/m_\E\Z)$.

	    Let $a = v_2(j)$, $b = v_2(m)$, $c = v_2(m_\E) = \max(1,b)$ and note that $m/2^b$ is odd and squarefree. We may write $\chi = \prod_{p \mid m} \chi_p$ where each $\chi_p$ is a nontrivial character of $(\Z/p\Z)^\times$ (for $p$ odd) or of $(\Z/2^{b}\Z)^\times$ (if $p = 2$). 

    If $S$ is the set of primes dividing $m_\E/j$ then, by Proposition~\ref{prop:Cejiff}, we need to show that $H \cap \Gamma_j$ contains an element whose image in each $GL_2(\Z/pj\Z)$ is nontrivial. We can write 
    \[\Gamma_j = K_{a}(2) \times \prod_{p \in S, p\ne 2} GL_2(\F_p).\]
     For $p$ odd, let $g_p \in SL_2(\F_p)$ with $g_p \ne e$. If $a = 0$, let $g_2 \in SL_2(\Z/2^{c}\Z)$ satisfy $\epsilon(g_2) = 1$ and $g_2 \bmod 2 \ne I$. If $a > 0$, let $g_2 \in K_a(2) \cap SL_2(\Z/2^c\Z)$ with $g_2 \not \in K_{a+1}(2)$. 
     
     In each case, we have $\chi_p(\det(g_p)) = 1$, while in the case $p = 2$ we also have $\epsilon(g_2) = 1$. The element $(g_{p})_{p \in S}$ is then an element of $H \cap \Gamma_j$ and not an element of any $\Gamma_{pj}$ for $p \in S$, as required.
\end{proof}

\begin{remark}
    One can show that, in the case of Serre curves, Theorem~\ref{T4b} applies unless $m_\E/j$ is odd (and > 1). Otherwise it does not apply, but the conclusion $\Cc_{\E,j} > 0$ still holds.
\end{remark}

\section{Coincidences}
\label{sec:coincidences}
We now give some examples where $\Cc_{\E,j} = 0$. Conjectures~\ref{conj:one-prime} and~\ref{conj:coinc23} predict that, in
any such case with $\E$ non-CM, $\Q(\E[j]) = \Q(\E[jp])$
for $p = 2$ or $3$ (which certainly implies that $\Cc_{\E,j} = 0$ by Corollary~\ref{cor:revdir}). These are examples of
`coincidences' of division fields in the sense of~\cite{DLZ}. For $p$ a prime, we call a \emph{$p$-coincidence} one of
the form $\Q(\E[j]) = \Q(\E[jp])$.

Given a $p$-coincidence $\Q(\E[d]) = \Q(\E[dp])$, we may construct further coincidences $\Q(\E[dk]) = \Q(\E[dkp])$ for all
$k > 1$ coprime to $p$. Call these coincidences \emph{imprimitive}, so that coincidences not arising in this way are
\emph{primitive}. There are simple constructions of elliptic curves with primitive coincidences
$\Q(\E[1]) = \Q(\E[2])$, $\Q(\E[2]) = \Q(\E[4])$ and $\Q(\E[2]) = \Q(\E[6])$, giving the following proposition.

\begin{prop}
    Given a positive integer $j$ with $12 \nmid j$, there exists an elliptic curve $\E$ with $\Cc_{\E,j}=0$.
\end{prop}
\begin{proof}
    If $j$ is odd then we may simply take a curve with $\Q(\E[2])=\Q$, such as $\E: y^2=x^3-7x+6$. Then $\Cc_{\E,j}=0$ by Corollary~\ref{cor:revdir} with $p=2$. 

    If $j$ is even and not divisible by 4 then take a curve with $\Q(\E[2]) = \Q(\E[4])$, such as $y^2 = x^3 + 13x - 34$ (see~\cite[Example~1.3]{DLZ}) and note that $\Q(\E[j]) = \Q(\E[2j])$ and so $\Cc_{\E,j} = 0$ by Corollary~\ref{cor:revdir} with $p=2$.
    
		    If $j$ is divisible by 4 then, by assumption, $3 \nmid j$. Let $\E$ be an elliptic curve with $\Q(\E[2])=\Q(\E[3])$, such as $y^2 = x^3 + 405x - 9882$ (see~\cite[Example~1.2]{DLZ}). 
	    Then $\Q(\E[3j])=\Q(\E[j])$ and by Corollary~\ref{cor:revdir} with $p=3$, we have $\Cc_{\E,j}=0$.
\end{proof}

In Remark~1.9 of~\cite{DLZ}, they speculate that the only coincidences $\Q(\E[m]) = \Q(\E[n])$
arise for $(n, m) \in \{(2, 3), (2, 4), (2, 6), (3, 6)\}$. This is not true, even if appropriately modified to take into
account the fact that one may replace $m$ and $n$ in a coincidence by their least common multiples with a third integer
$k$; we will see a number of examples below. However, we do make the following conjecture:

\begin{conj}\label{conj:coinc23} Suppose that $p \ge 5$. Then there are no non-CM elliptic curves over $\Q$ with a $p$-coincidence.
\end{conj}

\begin{rem}\label{rem:yvon} By~\cite[Corollary~3.10]{yvonCoincidencesDivisionFields2025}, if $\E$ is a non-CM elliptic curve over $\Q$
  with a $p$-coincidence then $p \mid 2 \Delta_\E$. Moreover, a standard argument shows that if there is a
  $p$-coincidence $\Q(\E[j]) = \Q(\E[jp])$ with $p > 5$ and $p \nmid j$, then $p \mid A(\E)$.
\end{rem}

Thanks to Zywina's algorithm~\cite{zywinaExplicitOpenImages2024}, the LMFDB now contains data on the images of mod $m$
Galois representations for elliptic curves over $\Q$. It is then straightforward to check\footnote{Code available at
  \url{https://github.com/jackshotton/invariant-factors-public}. Note that we only need to check non-CM elliptic curves
  with adelic index > 2, by Proposition~\ref{prop:serre}.} the following:

\begin{prop} Conjectures~\ref{conj:one-prime} and~\ref{conj:coinc23} are true for all curves in the
  LMFDB\footnote{As of 19th April 2026.}, and in particular for all curves of conductor $\le 500000$.
\end{prop}

\subsection{Constructions from abelian division fields.}
\label{sec:abelian}

We can construct some infinite families of primitive $p$-coincidences starting with certain abelian division fields. These
were classified in~\cite{gonzalez-jimenezEllipticCurvesAbelian2016}. The simplest is when we have an abelian
two-division field.

\begin{prop}\label{prop:ab-two-torsion}
  Suppose that $\Q(\E[2])$ is an abelian extension of $\Q$ with odd conductor $m$. That is, $\Q(\E[2])$ is either
  trivial, cubic, or a quadratic extension $\Q(\sqrt{m'})$ with $m' = 3 \bmod 4$ and $m = |m'|$.
    
    Then $\Q(\E[m]) = \Q(\E[2m])$.
\end{prop}
\begin{proof}
    By assumption, $\Q(\E[2]) \subset \Q(\zeta_m)\subset \Q(\E[m])$ with $m$ odd. Therefore $\Q(\E[m])=\Q(\E[2m])$.
\end{proof}

\subsubsection{Abelian $2$-power torsion}
\label{sec:abelian-2}
More families occur when we have abelian division fields $\Q(\E[2^r])$ such that $\Q(\E[2^{r-1}])$ contains
$\zeta_{2^r}$. By~\cite{gonzalez-jimenezEllipticCurvesAbelian2016}, we must have $r \le 3$ and $\Gal(\Q(\E[2^r])/\Q)$ is
an elementary abelian $2$-group.

\begin{prop}\label{prop:ab-four-torsion}
  \leavevmode
  \begin{enumerate}
  \item Suppose that $\E$ is an elliptic curve such that $\Gal(\Q(\E[4])/\Q)$ is abelian and that $\Q(\E[2]) = \Q(i)$. Suppose also
    that the conductor of $\Q(\E[4])/\Q$ is $4j$ with $j$ odd.

    Then $\Q(\E[2j]) = \Q(\E[4j])$.
  \item Suppose that $\E$ is an elliptic curve such that $\Gal(\Q(\E[8])/\Q)$ is abelian and that $\zeta_8 \in \Q(\E[4])$. Suppose 
    that the conductor of $\Q(\E[8])/\Q$ is $8j$ (and note that $j$ is necessarily odd as $\Gal(\Q(\E[8])/\Q) \cong
    C_2^k$ for some $k$).

    Then $\Q(\E[4j]) = \Q(\E[8j])$.
  \end{enumerate}
\end{prop}
\begin{proof}
  \begin{enumerate}
  \item We have that $\zeta_{2j} \in \Q(\E[2j])$ but also that $\zeta_4 \in \Q(\E[2])$. Thus
    \[\Q(\E[4]) \subset \Q(\zeta_{4j}) \subset \Q(\E[2j])\]
    and so $\Q(\E[2j]) = \Q(\E[4j])$.
  \item The proof is identical with $2$ replaced by $4$ and $4$ replaced by $8$.
  \end{enumerate}
\end{proof}

In~\cite{rouse_elliptic_2015} the possible $2$-adic Galois images for elliptic curves over $\Q$ are classified; in
~\cite{gonzalez-jimenezEllipticCurvesAbelian2016} these are used to enumerate elliptic curves with abelian $2^r$-division
field. For each conjugacy class of elementary abelian 2-group $H \subset GL_2(\Z/2^r\Z)$ there is an explicit modular
curve $X_H$ of genus 0 parametrising elliptic curves with $\im(\rho_{\E, 2^r}) \subset H$ (up to
conjugation)\footnote{There is a subtlety around quadratic twists, which we take care of below.}. The cases
in which this modular curve actually has rational points are listed in
~\cite[Tables~3~and~4]{gonzalez-jimenezEllipticCurvesAbelian2016}. Explicit generators for $H$ and parametrisations for
the curves may then be found at \url{https://users.wfu.edu/rouseja/2adic/}.

Suppose first that $r = 2$. To obtain subgroups $H$ such that elliptic curves $\E$ on $X_H$ satisfy the conditions of
Proposition~\ref{prop:ab-four-torsion} we must have that $\ker(H \to GL_2(\F_2)) = H\cap SL_2(\Z/4\Z)$. Checking each
possibility, the relevant modular curves are \texttt{X60d}, \texttt{X60}, \texttt{X27f}, \texttt{X27h},
\texttt{X27}. The notation is such that e.g.\ \texttt{X27f} corresponds to a subgroup $H$ not containing $\pm I$ and
\texttt{X27} corresponds to the subgroup with $\pm I$ added; the elliptic curves on \texttt{X27} will be quadratic
twists of those on \texttt{X27f}.

Curves on \texttt{X60d} (see~\cite{gonzalez-jimenezEllipticCurvesAbelian2016} or
\url{https://users.wfu.edu/rouseja/2adic/X60d.html}) have $\im(\rho_{\E, 4}) =\{I, \twomat{1}{1}{0}{-1}\}$ up to
conjugacy. Explicitly, we may take $\E$ given by
  \[y^2 = x^3 + (-432t^8 + 1512t^4 - 27)x + (3456t^{12} + 28512t^8 - 7128t^4 - 54)\]
  for $t \in \Q$ such that the discriminant is nonzero. In this case, $\Q(\E[4]) = \Q(\E[2]) = \Q(i)$, so this does not
  work for Proposition~\ref{prop:ab-four-torsion} unless $j = 1$; however, we can take quadratic twists (i.e. points on \texttt{X60}).

  For $\E = \E_t$ in this family, if $\E^{(\pm j)}$ is the quadratic twist of $\E$ by $\pm j$ for $j > 1$ odd and
  squarefree, then $\rho_{\E^{(\pm j)},4} = \chi_{\pm j}\rho_{\E,4}$, where $\chi_{\pm j}$ is the quadratic character
  associated to $\Q(\sqrt{\pm j})$, and the image of $\rho_{\E^{(\pm j)},4}$ is generated by
  \[
    \begin{pmatrix}
      1 & 1 \\
      0 & -1
    \end{pmatrix},
    \begin{pmatrix}
      -1 & 0 \\
      0 & -1
    \end{pmatrix}.
  \]
  For this subgroup, the kernel of reduction modulo $2$ coincides with the kernel of the determinant, and so
  $\Q(\E^{(\pm j)}[2]) = \Q(i)$. Moreover, $\Q(\E^{(\pm j)}[4]) = \Q(i,\sqrt{\pm j})$ has conductor $4j$. Therefore
  $\E^{(\pm j)}$ satisfies the hypotheses of Proposition~\ref{prop:ab-four-torsion}.

  For \texttt{X27h} there is similarly a parametrisation (see~\cite[Section~6.2]{gonzalez-jimenezEllipticCurvesAbelian2016})
  \[\E_t : y^2 = x^3 + (-432t^4 + 1512t^2 - 27)x + (3456t^6 + 28512t^4 - 7128t^2 - 54).\]
  Curves in this family have mod 4 Galois image contained in the subgroup generated by
\[
    \begin{pmatrix}
      1 & 1 \\
      0 & -1
    \end{pmatrix},
    \begin{pmatrix}
      1 & 2 \\
      0 & 1
    \end{pmatrix}
  \]
  and the hypotheses of Proposition~\ref{prop:ab-four-torsion} hold as long as we can arrange that the conductor of
  $\Q(\E_t[4])$ has the form $4j$ for $j$ odd. An explicit calculation\footnote{Code available at \url{https://github.com/jackshotton/invariant-factors-public}.} using
  ~\cite[Corollary~2.5]{gonzalez-jimenezEllipticCurvesAbelian2016} shows that $\Q(\E_t[4]) = \Q(i, \sqrt{t})$ and so we
  may take $t$ to be any odd integer. Curves on \texttt{X27} or \texttt{X27f} are quadratic twists of these by some
  squarefree integer $j$ and we will obtain similar examples as long as $j$ is odd.

  Now we suppose that $r = 3$. In this case we look for subgroups $H \subset GL_2(\Z/8\Z)$ such that
  $\ker(H \to GL_2(\Z/4\Z)) \subset SL_2(\Z/8\Z)$ and such that $H$ is an elementary abelian $2$-group. We
  check\footnote{Code available at \url{https://github.com/jackshotton/invariant-factors-public}.} each
  subgroup corresponding to a modular curve on the list in~\cite[Table
  4]{gonzalez-jimenezEllipticCurvesAbelian2016}. The subgroups of the form
  \texttt{X58?} may be disregarded, as in these cases $I + 4 M_2(\Z/8\Z) \subset H$ so our condition cannot hold. Of the
  remainder, our condition holds exactly for \texttt{X183e}, \texttt{X183k}, \texttt{X187e}, \texttt{X187i},
  \texttt{X189i} and \texttt{X189l}.

  We work through the case of \texttt{X187i}. From \url{https://users.wfu.edu/rouseja/2adic/187i} we have the
  parametrisation:
   \[
  \E_t : y^2 = x^3 + \left(-108t^{16} - 24192t^8 - 27648\right)x - 432t^{24} + 228096t^{16} + 3649536t^8 - 1769472.
  \]
  Setting $t = 1$ we obtain the elliptic curve with LMFDB label
  \href{https://www.lmfdb.org/EllipticCurve/Q/960/a/5}{960.a5}. This has equation
  \[
    \E : y^2 = x^3 - x^2 - 641x + 3105 = (x-23)(x-5)(x+27).
  \]
  The pairwise differences of the roots are $18$, $50$ and $32$, so
  ~\cite[Corollary~2.5]{gonzalez-jimenezEllipticCurvesAbelian2016} gives
  \[
    \Q(\E[4]) = \Q(i,\sqrt{2}) = \Q(\zeta_8).
  \]
  For this curve one finds that $\Q(\E[8])$ is abelian of conductor $120 = 8\cdot 15$, and hence by Proposition~\ref{prop:ab-four-torsion} we have
  \[
    \Q(\E[60]) = \Q(\E[120]).
  \]

\subsubsection{Abelian $3$-torsion}

We may also construct examples starting from abelian $3$-division fields. These must be quadratic or biquadratic and in
each case there is a simple parametrisation. To obtain coincidences, we exploit the fact that the adelic image is always
contained in the Serre subgroup (see the proof of Proposition~\ref{prop:serre}), the kernel of a particular quadratic
character.

\begin{prop}\label{prop:min-three-torsion} Suppose that $\Q(\E[3])= \Q(\zeta_3)$ and that $3 \mid \Delta_\E^{sf}$, the
  squarefree part of the discriminant of $\E$. Let
\[3m = \begin{cases} 2 |\Delta_\E^{sf}| & \text{if $\Delta_\E^{sf} \equiv 1 \bmod 4$} \\
4 |\Delta_\E^{sf}| & \text{otherwise.}
    \end{cases}\]
Then $\Q(\E[m]) = \Q(\E[3m])$.
\end{prop}
\begin{proof}
  Let $\chi$ be the primitive Dirichlet character associated to the quadratic extension $\Q(\sqrt{\Delta_\E})$ and let
  $\epsilon$ be as in the proof of Proposition~\ref{prop:serre}. We may factor $\chi$ as $\chi_3 \chi'$ where $\chi_3$ is
  the quadratic character of conductor 3 and $\chi'$ is a quadratic character of conductor coprime to 3. Then the
  character $\epsilon \cdot \chi\circ \det$ of $GL_2(\Z/m_\E\Z)$ is trivial on the image $H$ of $\rho \bmod m_\E$. If
  $h \in H \cap \Gamma_m$, then $\chi'(\det(h)) = 1$ and $\epsilon(h) = 1$. It follows that $\chi_3(\det(h)) = 1$, so
  that $h$ fixes $\zeta_3$. Thus
    \[\Q(\E[3]) = \Q(\zeta_3) \subset \Q(\E[m])\]
    and so $\Q(\E[m]) = \Q(\E[3m])$, as required.
\end{proof}

\begin{remark}\label{rem:hesse}
  As in~\cite[Section~6.2]{gonzalez-jimenezEllipticCurvesAbelian2016}, the curves with
  $\Q(\E[3])=\Q(\zeta_3)$ are parametrised by the family
  \[
    \E_t : y^2 = x^3 - 27t(t^3+8)x + 54(t^6-20t^3-8).
  \]
  where $t \in \Q$ and $t \ne 1$. This is the Weierstrass form of the Hesse plane cubic
  \[
    x^3 + y^3 + z^3 = 3txyz.
  \]
  As in
  ~\cite{joyeHessianEllipticCurves2001}, the discriminant of this model is $2^{12}3^9(t^3-1)^3$. If we choose $t$ so that
  $v_3(t^3 - 1)$ is even (for instance, $t = 3u$ for $u \in \Z$), then $3 \mid \Delta_{\E_t}^{sf}$ and $\E_t$ satisfies the hypotheses of
  Proposition~\ref{prop:min-three-torsion}.
\end{remark}

\begin{remark}\label{rem:hesse-twist}
  Suppose that $\E$ and $m$ are as in Proposition~\ref{prop:min-three-torsion} and that $d$ is squarefree with
  $\Q(\sqrt{d}) \subset \Q(\zeta_m)$. If $\E^{(d)}$ is the quadratic twist of $\E$ by $d$ then
  $\Q(\E^{(d)}[3]) = \Q(\zeta_3, \sqrt{d})$ and we still have
  \[\Q(\E^{(d)}[m]) = \Q(\E^{(d)}[3m])\]
  by exactly the same argument as in the proof of Proposition~\ref{prop:min-three-torsion}, noting that
  $\Delta^{sf}_{\E^{(d)}} = \Delta^{sf}_{\E}$.

  We give a numerical example: let $t = 2$ in the above parametrisation and $\E$ be the resulting curve. We obtain the
  elliptic curve with LMFDB label \href{https://www.lmfdb.org/EllipticCurve/Q/189/b/2}{189.b2} which has discriminant
  $3^9\cdot 7^3$ and so $\Q(\E[14]) = \Q(\E[42])$ by Proposition~\ref{prop:min-three-torsion}. By
  Remark~\ref{rem:hesse-twist}, the same holds with $\E$ replaced by its quadratic twist $\E^{(7)}$ (with LMFDB label
  \href{https://www.lmfdb.org/EllipticCurve/Q/1323/h/3}{1323.h3}).
\end{remark}

\subsection{Smorgasbord}
\label{sec:smorgasbord}

We finish with a small selection of examples of coincidences that we found interesting.

\begin{example}
\begin{enumerate}
    \item There are numerous examples in which $\Q(\E[2])$ is quadratic and contained in multiple division fields
      $\Q(\E[j])$ for $j$ odd. For example, the elliptic curve with LMFDB label
      \href{https://www.lmfdb.org/EllipticCurve/Q/198/c/1}{198.c1} has
    \[\Q(\E[2]) = \Q(\sqrt{33}) \subset \Q(\E[33]) \cap \Q(\E[55]).\]
  \item The elliptic curve with LMFDB label \href{https://www.lmfdb.org/EllipticCurve/Q/196/a/1}{196.a1} has $\Q(\E[2])$
    cubic of conductor 7 and therefore \[\Q(\E[2]) = \Q(\zeta_7)^+ \subset \Q(\E[7])\] as in
    Proposition~\ref{prop:ab-two-torsion}. However, in this case we also have $\Q(\E[2]) \subset \Q(\E[9])$ and
    therefore $\Q(\E[9]) = \Q(\E[18])$. In particular, $\Cc_{\E, 7} = \Cc_{\E, 9} = 0$, showing that $\Cc_{\E,j}$ can
    vanish for two coprime values of $j$.
    \item The elliptic curve with LMFDB label \href{https://www.lmfdb.org/EllipticCurve/Q/300/b/1}{300.b1} has
      $\Q(\E[2]) \subset \Q(\E[9])$ and $\Q(\E[2])$ nonabelian.
    \item  The elliptic curve with LMFDB label \href{https://www.lmfdb.org/EllipticCurve/Q/300/b/2}{300.b2} has 
    \[\Q(\E[2]) \subset \Q(\E[3]) \subset \Q(\E[10]).\]
    Thus $\Q(\E[3]) = \Q(\E[6])$ and $\Q(\E[10])= \Q(\E[30])$. Note that in this case the mod $3$ representation has
    image of order 12, and so is nonabelian; in particular this curve has a $3$-coincidence that is not explained by
    Proposition~\ref{prop:min-three-torsion}.
   \item The elliptic curve with LMFDB label \href{https://www.lmfdb.org/EllipticCurve/Q/1680/g/2}{1680.g2} has 
    $\Q(\E[4]) \subset \Q(\E[105])$
    and so \[\Q(\E[105]) = \Q(\E[210]) = \Q(\E[420]).\] Here $\Gal(\Q(\E[4])/\Q)$ is elementary abelian of order 8. 
    \end{enumerate}
\end{example}

\section{\texorpdfstring{Positivity of $\Cc_{\E,1}$}{Positivity of C(E,1)}}
\label{sec:CoMu}

The purpose of this section is to explain and repair a small error in the proof of Corollary~\ref{cor:CE1} given in~\cite[Section~6]{CoMu}.
The following lemma, which is the key to the argument, is proved in~\cite{CoMu} via an inclusion-exclusion argument.

\begin{lemma} [\cite{CoMu}, Lemma 6.1] \label{C6.1}
    Let $\mathcal{F}=(F_q)_{q \in \mathcal{P}}$, $\mathcal{F}'=(F'_q)_{q \in \mathcal{P}'}$ be two families of finite Galois extensions of a number field $F$, indexed by sets of rational primes $\mathcal{P}' \subset \mathcal{P}$. For any square-free integer $k$ composed of primes from $\mathcal{P}$, define:
    \begin{itemize}
        \item fields $F_k = \prod_{q|k, q \in \mathcal{P}} F_q$ and  $F'_k = \prod_{q|k,  q \in \mathcal{P}'} F'_q $, where the product denotes compositum;
        \item integers $N_k = [F_k:F]$ and $N'_k = [F'_k:F]$ where $N_1=N'_1=1$;
        \item quantities
        \[\delta(\mathcal{F}) = \sum_{\substack{k \\ q|k \Rightarrow q \in \mathcal{P}}} \frac{\mu(k)}{N_k}, \quad \delta(\mathcal{F}') = \sum_{\substack{k \\ q|k \Rightarrow q \in \mathcal{P}'}} \frac{\mu(k)}{N'_k}.\]
    \end{itemize}
    In addition, suppose that:
    \begin{itemize}
        \item $\mathcal{F}$ \emph{covers} $\mathcal{F}'$. That is: for all $q' \in \mathcal{P}'$, there exists $q \in \mathcal{P}$ such that $F'_{q'} \subseteq F_q$, and for all $q \in \mathcal{P}$, there exists $q' \in \mathcal{P}'$ such that $F'_{q'} \subseteq F_q$;
        \item The sums defining $\delta(\mathcal{F})$ and $\delta(\mathcal{F}')$ are absolutely convergent.
    \end{itemize}
    Then $\delta(\mathcal{F}) \ge \delta(\mathcal{F}')$.
\end{lemma}

\begin{cor} \label{C2b}
    If, in addition to the hypotheses in Lemma~\ref{C6.1}, the fields $(F'_q)_{q \in \mathcal{P}'}$ are linearly disjoint over $F$ then
    \[
    \delta(\mathcal{F}) \ge \delta(\mathcal{F}') =  \prod_{q \in \mathcal{P}'} \left(1-\frac{1}{N'_q}\right).
    \]
\end{cor}
In~\cite{CoMu} this is applied as follows. Let $\E$ be a non-CM elliptic curve over $\Q$ with $\Q(\E[2]) \ne \Q$ and let $F_q = \Q(\E[q])$ for each prime $q$ (with $\mathcal{P}$ the set of all primes). Let $K_2 \subset F_2$ be the largest abelian extension of $\Q$ contained in $F_2$, which is either quadratic or cubic by assumption. Let 
\[\mathcal{P}' = \{q \in \mathcal{P} : \Q(\zeta_q) \cap K_2 = \Q\}\]
and, for $q \in \mathcal{P}'$, define 
\[F'_q = \begin{cases} F_q & \text{if $q \nmid m_\E$} \\ 
\Q(\zeta_q) & \text{if $q \mid m_\E$, $q \ne 2$} \\
K_2 & \text{if $q = 2$}.
\end{cases}\]
Then with $\Fc = \{F_q\}_{q \in \Pc}$ and $\Fc' = (F'_q)_{q \in \Pc'}$ we have that $\Fc$ covers $\Fc'$. It is asserted in~\cite{CoMu} that the fields $\Fc'$ are linearly disjoint, and it would then follow from Corollary~\ref{C2b} that \[\Cc_{\E,1} = \delta(\Fc) > 0.\] However, although the fields $F'_q$ are \emph{pairwise} linearly disjoint, they are \emph{not} linearly disjoint if (for example) the conductor of $\Q(\sqrt{\Delta_\E})$ is composite and odd. This is the gap in~\cite[Section 6]{CoMu}. It may be fixed by making a closer examination of $\delta(\Fc')$, which we now do.

    Define 
    \[R = \prod_{\substack{q \mid m_\E\\ q \ne 2\\ K_2 \cap \Q(\zeta_q) = \Q}} q\] and define $S$ to be the minimal integer dividing $R$ such that $K_2 \subset \mathbb{Q}(\zeta_S)$ (if there is no such integer $S$, then the families $\mathcal{F}$ and $\mathcal{F}'$ as defined above fulfil the conditions of Corollary~\ref{C2b} and we argue as in~\cite{CoMu}).
    
   Let $d=[K_2:\mathbb{Q}] = 2$ or $3$. Then, for $k|R$, we have, $N'_k=\phi(k)$ and 
   \[N'_{2k}= \begin{cases}d \phi(k) & \text{if $S \nmid k$}\\ \phi(k) &  \text{if $S|k$.}\end{cases}\] Thus
    \begin{align*}
        \sum_{k|R}\mu(k)\left(\frac{1}{N'_k}-\frac{1}{N'_{2k}} \right)&=\sum_{k|R}\mu(k)\left(\frac{1}{\phi(k)}-\frac{1}{d\phi(k)}\right)-\sum_{k|\frac{R}{S}}\mu(Sk)\left(\frac{1}{\phi(Sk)}-\frac{1}{d\phi(Sk)}\right) \\
        &=\left(1-\frac{1}{d}\right)\left(\prod_{q|R }\left(1-\frac{1}{\phi(q)}\right)-\frac{\mu(S)}{\phi(S)}\prod_{q|\frac{R}{S} }\left(1-\frac{1}{\phi(q)}\right)\right) \\
        &= \frac{1}{\phi(S)}\left(1-\frac{1}{d}\right)\prod_{q|\frac{R}{S} }\left(1-\frac{1}{\phi(q)}\right)\left(\prod_{q|S }(\phi(q)-1)-\mu(S)\right) \\
        &>0
    \end{align*}
    since $S \ge 3$. Therefore
    \[\delta(\Fc') = \sum_{\substack{k \\ q|k \Rightarrow q \in \mathcal{P}'}} \frac{\mu(k)}{N'_k}=\sum_{k|R}\mu(k)\left(\frac{1}{N'_k}-\frac{1}{N'_{2k}} \right)\prod_{q \: \nmid \, m_{\E}} \left(1-\frac{1}{N'_q}\right)>0.\]
    Therefore, by Lemma~\ref{C6.1}, 
\[\mathcal{C}_{\E,1}=\delta(\mathcal{F})\ge\delta(\mathcal{F}')> 0\]
as required.

\printbibliography
\end{document}